\documentclass{article}

\usepackage{myarxiv}
\usepackage[utf8]{inputenc} 
\usepackage[T1]{fontenc}    
\usepackage{hyperref}       
\usepackage{url}            
\usepackage{booktabs}       
\usepackage{amsfonts}       
\usepackage{nicefrac}       
\usepackage{microtype}      
\usepackage{lipsum}

\usepackage{amssymb,amsmath,amsfonts,amsthm,enumerate}
\usepackage{subfigure}
\usepackage{graphicx}
\usepackage{epsfig}
\usepackage{float}
\usepackage[usenames]{color}
\usepackage{cite}
\usepackage{bm}
\usepackage{enumitem}
\usepackage{empheq}
\usepackage{dsfont}
\setlist{nosep}

\newtheorem{theorem}{Theorem}[section]
\newtheorem{lemma}[theorem]{Lemma}
\newtheorem{proposition}[theorem]{Proposition}
\newtheorem{corollary}[theorem]{Corollary}
\newtheorem{definition}[theorem]{Definition}
\newtheorem{problem}[theorem]{Problem}
\numberwithin{equation}{section}

\def\BC{\mathbb C}
\def\BN{\mathbb N}
\def\BR{\mathbb R}
\def\cA{\mathcal A}
\def\cB{\mathcal B}
\def\cD{\mathcal D}

\def\cL{\mathcal L}
\def\cM{\mathcal M}
\def\cN{\mathcal N}
\def\cO{\mathcal O}

\def\I{\mathrm I}

\def\rd{\mathrm d}
\def\rRe{\mathrm{Re}}

\def\rdiv{\mathrm{div}}
\def\e{\mathrm e}

\def\ri{\mathrm i}
\def\rin{\mathrm{in}}
\def\out{\mathrm{out}}

\def\rspan{\mathrm{span}}
\def\supp{\mathrm{supp}}

\def\Ga{\Gamma}
\def\De{\Delta}
\def\Te{\Theta}
\def\La{\Lambda}

\def\Om{\Omega}
\def\al{\alpha}
\def\be{\beta}
\def\ga{\gamma}
\def\de{\delta}

\def\ve{\varepsilon}
\def\te{\theta}
\def\ze{\zeta}
\def\ka{\kappa}
\def\la{\lambda}
\def\si{\sigma}
\def\vp{\varphi}
\def\om{\omega}
\def\f{\frac}
\def\nb{\nabla}
\def\ov{\overline}
\def\pa{\partial}
\def\wh{\widehat}
\def\wt{\widetilde}
\def\tri{\triangle}

\title{Unique determination of several coefficients\\
in a fractional diffusion(-wave) equation\\
by a single measurement}

\author{
  Yavar \uppercase{Kian}\\
  Aix-Marseille Universit\'e, Universit\'e de Toulon, CNRS, CPT, Marseille, France.\\
  \texttt{yavar.kian@univ-amu.fr}\\
  \And
  Zhiyuan \uppercase{Li}\\
  School of Mathematics and Statistics, Shandong University of Technology\\
  Zibo, Shandong 255049, China.\\
  \texttt{zyli@sdut.edu.cn}\\
  \And
  Yikan \uppercase{Liu}\\
  Graduate School of Mathematical Sciences, The University of Tokyo\\
  3-8-1 Komaba, Meguro-ku, Tokyo 153-8914, Japan.\\
  \texttt{ykliu@ms.u-tokyo.ac.jp}\\
  \And
  Masahiro \uppercase{Yamamoto}\\
  Graduate School of Mathematical Sciences, The University of Tokyo\\
  3-8-1 Komaba, Meguro-ku, Tokyo 153-8914, Japan;\\
  Honorary Member of Academy of Romanian Scientists\\
  Splaiul Independentei Street, No.\! 54, 050094 Bucharest, Romania;\\
  Peoples' Friendship University of Russia (RUDN University)\\
  6 Miklukho-Maklaya Street, Moscow 117198, Russian Federation.\\
  \texttt{myama@ms.u-tokyo.ac.jp}\\
}
\newcommand{\shortauthor}{Y. Kian, Z. Li, Y. Liu, M. Yamamoto}

\begin{document}
\maketitle

\begin{abstract}
We consider the inverse problem of determining different type of information about a diffusion process, described by ordinary or fractional diffusion equations stated on a bounded domain, like the density of the medium or the velocity field associated with the moving quantities from a single boundary measurement. This properties will be associated with some general class of time independent coefficients that we recover from a single Neumann boundary measurement, on some parts of the boundary, of the solution of our diffusion equation with a suitable boundary input, located on some parts of the boundary.
\end{abstract}

\MRsubject{35R30\and 35R11\and 58J99}

\section{Introduction}\label{sec-intro}

Let $T\in\BR_+:=(0,\infty)$ and $\al\in(0,2)$. By $\pa_t^\al$ we denote the $\al$-th order Caputo derivative with respect to $t$ defined by
\[
\pa_t^\al f(t):=\left\{\!\begin{alignedat}{2}
& \f1{\Ga(\lceil\al\rceil-\al)}\int_0^t\f{f^{(\lceil\al\rceil)}(s)}{(t-s)^{\al-\lfloor\al\rfloor}}\,\rd s, & \quad & \al\in(0,2)\setminus\{1\},\\
& f'(t), & \quad & \al=1,
\end{alignedat}\right.
\]
where $\lfloor\cdot\rfloor$ and $\lceil\cdot\rceil$ denote floor and ceiling functions, respectively. Let $\Om\subset\BR^d$ ($d=2,3,\ldots$) be a bounded connected domain with a $C^{1,1}$ boundary $\pa\Om$. Let $\rho\in L^\infty(\Om)$, $a\in C^1(\ov\Om)$, $\bm B=(B_1,\ldots,B_d)\in(L^\infty(\Om))^d$, $c\in L^q(\Om)$ ($q\in(d,\infty]$) and assume that
\begin{equation}\label{ell}
\underline\rho\le\rho\le\ov\rho\mbox{ in }\Om,\quad a>0\mbox{ on }\ov\Om\,,\quad c\ge0\mbox{ in }\Om,
\end{equation}
where $\underline\rho\,,\ov\rho\in\BR_+$ are constants. The first object of this paper is the following initial-boundary value problem for a time-fractional diffusion(-wave) equation with a nonhomogeneous Dirichlet boundary condition
\begin{equation}\label{eq1}
\begin{cases}
\rho\,\pa_t^\al u-\rdiv\left(a\nb u\right)+\bm B\cdot\nb u+c\,u=0 & \mbox{in }(0,T)\times\Om,\\
\begin{cases}
u=0 & \mbox{if }0<\al\le1,\\
u=\pa_t u=0 & \mbox{if }1<\al<2
\end{cases} & \mbox{in }\{0\}\times\Om,\\
u=\Phi & \mbox{on }(0,T)\times\pa\Om.
\end{cases}
\end{equation}
Meanwhile, we will also consider the same problem as \eqref{eq1} on manifolds. To this end, let $(\cM,g)$ be a smooth compact connected Riemannian manifold of dimension $d\ge2$ with a boundary $\pa\cM$. Let $\mu\in C^\infty(\cM)$ and $c\in L^\infty(\cM)$ such that $\mu>0$ and $c\ge0$ on $\cM$. We introduce the weighted Laplace-Beltrami operator
\[
\tri_{g,\mu}:=\mu^{-1}\rdiv_g\,\mu\nb_g,
\]
where $\rdiv_g$ and $\nb_g$ denote divergence and gradient operators on $(\cM,g)$ respectively.
The second object of this paper is the following initial-boundary value problem on the manifold $\cM$
\begin{equation}\label{eqM1}
\begin{cases}
\pa_t^\al u-\tri_{g,\mu}u+c\,u=0 & \mbox{in }(0,T)\times\cM,\\
\begin{cases}
u=0 & \mbox{if }0<\al\le1,\\
u=\pa_t u=0 & \mbox{if }1<\al<2
\end{cases} & \mbox{in }\{0\}\times\cM,\\
u=\Phi & \mbox{on }(0,T)\times\pa\cM.
\end{cases}
\end{equation}

This article is concerned with the following coefficient inverse problem for \eqref{eq1} and \eqref{eqM1}.

\begin{problem}\label{prob-CIP}
Let $u$ satisfy \eqref{eq1} (or \eqref{eqM1}) and $\Ga_\rin,\Ga_\out$ be two open subsets of $\pa\Om$ (or $\pa\cM$). For a suitably chosen Dirichlet boundary input $\Phi$ supported on the sub-boundary $[0,T]\times\Ga_\rin$, determine simultaneously the coefficients $(\al,\rho,a,\bm B,c)$ (or $(\al,c,\cM)$) from a single measurement of $a\,\pa_\nu u$ on the sub-boundary $(0,T)\times\Ga_\out$, where $\nu$ denotes the outward unit normal vector to $\pa\Om$ (or $\pa\cM$).
\end{problem}

Let us mention that time-fractional diffusion(-wave) equations of the form \eqref{eq1} and \eqref{eqM1} describe diffusion of different kinds of physical phenomena. For $\al\ne1$, \eqref{eq1} and \eqref{eqM1} describe the anomalous diffusion of substances in heterogeneous media, diffusion in inhomogeneous anisotropic porous media, turbulent plasma, diffusion in a turbulent flow, a percolation model in porous media, several biological and financial problems (see \cite{CSLG}). For instance, it is known (see \cite{AG}) that in several context the classical diffusion-advection equation is not a suitable model for describing field data of diffusion of substances in the soil. In this context, time-fractional diffusion(-wave) equations are regarded as an alternative model. Note also that time-fractional diffusion(-wave) equations are derived from continuous-time random walk (see \cite{MK,RA}). Due to their modeling feasibility, time-fractional differential equations have received considerable attention in the last decades. Without being exhaustive we refer to \cite{L,MR,SM,P} for further details. Especially, the well-posedness for problem \eqref{eq1} has been studied by \cite{SY,KY} for $\bm B=0$ and by \cite{JLLY} for $\bm B\ne0$.

On the other hand, the above inverse problem addressed in the present paper corresponds to the determination of several parameters describing the diffusion of some physical quantities. The convection term $\bm B$ is associated with the velocity field of the moving quantities, while the coefficients $(\rho,a,c)$ can be associated to the density of the medium. For instance, our inverse problem can be stated as the determination of the velocity field and the density of the medium in the diffusion process of a contaminant in a soil from a single measurement at $\Ga_\out$.

In retrospect, many authors considered inverse problems for \eqref{eq1} when $\al=1$. Without being exhaustive, we can refer to \cite{CK1,CaK,Cho,ChK,Katchalov2004}. In particular, we mention the work of \cite{AS} where the recovery of the coefficient $c$ has been addressed from a single boundary measurement. The idea of \cite{AS} is to use a suitable input $\Phi$ that allows to recover the associated elliptic Dirichlet-to-Neumann map from a single measurement of the solution of the associated parabolic equation on the lateral boundary $(0,T)\times\pa\Om$. This idea has been extended by \cite{CY} to the recovery of the convection term $\bm B$ in the case of $d=2$. Here the author apply the results \cite{CY1,CY2} related to the recovery of a convection term appearing in a stationary convection diffusion equation from the associated Dirichlet-to-Neumann map.

In contrast to $\al=1$, inverse problems associated with \eqref{eq1} when $\al\in(0,2)\setminus\{1\}$ have received less attention. For $d=1$, \cite{CNYY} determined uniquely the fractional order $\al$ and a coefficient from Dirichlet boundary measurements. In \cite{FK,SY}, the authors considered the problem of stably determining a time-dependent parameter appearing in a time-fractional diffusion equation. In \cite{LIY}, the authors determine uniquely coefficients by mean of the Dirichlet-to-Neumann map associated with the system applied to Dirichlet boundary conditions taking the form $\la(t)\Phi(\bm x)$, where $\la$ is known. In \cite{KOSY}, the authors considered the recovery of a general class of coefficients on a Riemannian manifold and a Riemannian metric from the partial Dirichlet-to-Neumann map, associated with the fractional diffusion equation under consideration, taken at one arbitrarily fixed time. In this work, the authors start by proving the recovery of boundary spectral data associated with the elliptic part of their equation. Then using inverse spectral results and related inverse problems stated in \cite{CK1,CK2,KKL,LO} they complete their uniqueness results. For variable order and distributed order fractional diffusion equations we mention the work of \cite{KSY,LKS} where results related to inverse problems for these equations have been stated. In the work \cite{KY1}, the authors proved the reconstruction of source terms and the stable recovery of some class of zero order time dependent coefficients appearing in fractional diffusion equation on a cylindrical domain. Finally, we mention the recent work of \cite{HLYZ} where the recovery of a Riemannian manifold without boundary has been proved from a single internal measurement of the solution of a fractional diffusion equation with a suitable internal source.

The reminder of this paper is organized as follows. In Sections \ref{sec-main}, we recall the necessary ingredients to treat Problem \ref{prob-CIP} and state the main uniqueness results along with some explanations and remarks. As a key to establishing the main results, in Section \ref{sec-analytic} we prove the time-analyticity of the solutions to the forward problems under consideration. Proofs of the main results is provided in Section \ref{sec-proof}.

\section{Preliminaries and main results}\label{sec-main}

We start by fixing the notations and terminology used in the sequel. Throughout this paper, by $\BN:=\{1,2,\ldots\}$ we denote the positive natural numbers. Let $H^m(\Om)$, $H^{m-1/2}(\pa\Om)$, $W^{m,q}(\Om)$, etc.\! ($m\in\mathbb Z$, $q\in[1,\infty]$) denote the usual Sobolev spaces (see Adams \cite{A75}). Given Banach spaces $X$ and $Y$, by $\cB(X,Y)$ we denote the family of bounded linear operators from $X$ to $Y$. For a connected set $\cO\subset\BR$ or $\BC$, we denote by $C^\om(\cO;X)$ the family of analytic functions in $\cO$ taking values in $X$.

\subsection{Statements of the Main results}

We first deal with the initial-boundary value problem \eqref{eq1} in a bounded domain $\Om\subset\BR^d$. We denote the inner products of $L^2(\Om)$ and $L^2(\pa\Om)$ by $(\,\cdot\,,\,\cdot\,)$ and $\langle\,\cdot\,,\,\cdot\,\rangle$, respectively, that is,
\[
(f_1,f_2):=\int_\Om f_1(\bm x)f_2(\bm x)\,\rd\bm x,\quad f_1,f_2\in L^2(\Om);\qquad\langle f_3,f_4\rangle:=\int_{\pa\Om}f_3(\bm y)f_4(\bm y)\,\rd\si(\bm y),\quad f_3,f_4\in L^2(\pa\Om).
\]
For $f\in L^1_{\mathrm{loc}}(\BR_+)$, we denote its Laplace transform as
\[
\wh f(\xi)=(\cL f)(\xi):=\int_0^\infty\e^{-\xi t}f(t)\,\rd t.
\]
Recall that, according to \cite{SY,KY}, we can define the weak solutions of \eqref{eq1} in the following way.

\begin{definition}\label{d1}
Let $F\in L^1(0,T;L^2(\Om))$. We say that the problem
\begin{equation}\label{eqq1}
\begin{cases}
\rho\,\pa_t^\al u-\rdiv\left(a\nb u\right)+c\,u=F & \mbox{in }(0,T)\times\Om,\\
\begin{cases}
u=0 & \mbox{if }0<\al\le1,\\
u=\pa_t u=0 & \mbox{if }1<\al<2
\end{cases} & \mbox{in }\{0\}\times\Om,\\
u=0 & \mbox{on }(0,T)\times\pa\Om
\end{cases}
\end{equation}
admits a weak solution $u$ if there exists $v\in L^\infty_{\mathrm{loc}}(\BR_+;L^2(\Om))$ such that
\begin{enumerate}
\item[{\rm(i)}]$v|_{(0,T)\times\Om}=u$ and $\inf\{\ve>0\mid\e^{-\ve t}v(t,\,\cdot\,)\in L^1(\BR_+;L^2(\Om))\}=0$, 
\item[{\rm(ii)}] for all $\xi>0$, the Laplace transform $\wh v(\xi,\,\cdot\,)$ of $v(t,\,\cdot\,)$ solves
\[
\left\{\!\begin{alignedat}{2}
& -\rdiv(a\nb\wh v(\xi,\,\cdot\,))+(\xi^\al\rho+c)\wh v(\xi,\,\cdot\,)=\int_0^T\e^{-\xi t}F(t,\,\cdot\,)\,\rd t & \quad & \mbox{in }\Om,\\
& \wh v(\xi,\,\cdot\,)=0 & \quad & \mbox{on }\pa\Om.
\end{alignedat}\right.
\]
\end{enumerate}
\end{definition}

Following \cite{KY,KSY,SY}, there exists $S\in L^1(0,T;\cB(L^2(\Om);H^1(\Om)))$ such that the solution of \eqref{eqq1} takes the form
\[
u(t,\,\cdot\,)=\int_0^t S(t-s)F(s,\,\cdot\,)\,\rd s.
\]
Therefore, regarding the advection term $\bm B\cdot\nb u$ as an additional source term, we define the solution of the problem
\[
\begin{cases}
\rho\,\pa_t^\al u-\rdiv\left(a\nb u\right)+\bm B\cdot\nb u+c\,u=F & \mbox{in }(0,T)\times\Om,\\
\begin{cases}
u=0 & \mbox{if }0<\al\le1,\\
u=\pa_t u=0 & \mbox{if }1<\al<2
\end{cases} & \mbox{in }\{0\}\times\Om,\\
u=0 & \mbox{on }(0,T)\times\pa\Om
\end{cases}
\]
in the mild sense as the solution of the integral equation
\[
u(t,\,\cdot\,)=\int_0^t S(t-s)F(s,\,\cdot\,)\,\rd s-\int_0^t S(t-s)\bm B\cdot\nb u(s,\,\cdot\,)\,\rd s.
\]

For the existence and uniqueness of solutions of \eqref{eq1} in the above sense as well as classical properties of solutions of such problems, we refer to \cite{JLLY,KY,KSY,SY}. In the case $\al=1$, the solution of \eqref{eq1} corresponds to the classical variational solution of this parabolic equation lying in $H^1(0,T;H^{-1}(\Om))\cap L^2(0,T;H^1(\Om))$.

Now we turn to the investigation of Problem \ref{prob-CIP}. For the choice of open sub-boundaries $\Ga_\rin$ and $\Ga_\out$, we assume that
\begin{equation}\label{t1a}
\Ga_\rin\cup\Ga_\out=\pa\Om,\quad\Ga_\rin\cap\Ga_\out\ne\emptyset.
\end{equation}
This condition will be relaxed later in the framework of smooth Riemannian manifolds with smooth coefficients.

Next we specify the choice of the Dirichlet input $\Phi$, which plays an essential role in the consideration of Problem \ref{prob-CIP}. Let $\chi\in C^\infty(\pa\Om)$ satisfy $\supp\,\chi\subset\Ga_\rin$ and $\chi=1$ on $\Ga'_\rin$, where $\Ga'_\rin$ is an open subset of $\pa\Om$ such that $\Ga'_\rin\cup\Ga_\out=\pa\Om$ and $\Ga'_\rin\cap\Ga_\out\ne\emptyset$. We fix $T_0\in(0,T]$ and choose a strictly increasing sequence $\{t_k\}_{k=0}^\infty$ such that $t_0=0$ and $\lim_{k\to\infty}t_k=T_0$. Consider $\{p_k\}_{k=0}^\infty$ a sequence of $\BR_+$ and a sequence $\{\psi_k\}_{k\in\BN}\subset C^\infty(\BR_+;\BR_+)$ such that
\[
\psi_k=\begin{cases}
0 & \mbox{on }(0,t_{2k-2}],\\
p_k & \mbox{on }[t_{2k-1},\infty).
\end{cases}
\]
We fix also  $\{b_k\}_{k=0}^\infty$ a sequence of $\BR_+$ such that
\[
\sum_{k=1}^\infty \left[b_k\|\psi_k\|_{W^{2,\infty}(\BR_+)}\right]<\infty.
\]
Finally, we select a sequence $\{\eta_k\}_{k\in\BN}\subset H^{3/2}(\pa\Om)$ such that $\rspan\{\eta_k\}$ is dense in $H^{3/2}(\pa\Om)$ and $\|\eta_k\|_{H^{3/2}(\pa\Om)}=1$ ($k\in\BN$). Now we can construct the input $\Phi\in C^2(\BR_+;H^{3/2}(\pa\Om))$ as
\begin{equation}\label{Phi}
\Phi(t,\bm x):=\chi(\bm x)\sum_{k=1}^\infty b_k\,\psi_k(t)\eta_k(\bm x).
\end{equation}
Note that clearly, we have $\supp\,\Phi\subset\BR_+\times\Ga_\rin$. Using this definition of the input $\Phi$ we can now state our first two main results in the Euclidean case.

First, we fix $\al\in(0,2)$, $\bm B=\bm0$ and consider the recovery of the coefficients $(\rho,a,c)$. Actually, due to the obstruction described in \cite{CK2,KOSY}, we can only consider the recovery of any two among the three coefficients $\rho,a,c$. Our first main result can be stated as follows.

\begin{theorem}\label{t1}
Let $\al\in(0,2)$ be fixed, $\Ga_\rin,\Ga_\out\subset\pa\Om$ satisfy \eqref{t1a} and the triples $(\rho^j,a^j,c^j)\ (j=1,2)$ fulfill \eqref{ell}. Suppose that either of the following conditions are satisfied:
\begin{enumerate}
\item[{\rm(i)}]$\rho^1=\rho^2$ and
\begin{equation}\label{t1b}
\nb a^1=\nb a^2\quad\mbox{on }\pa\Om.
\end{equation}
\item[{\rm(ii)}]$a^1=a^2$ and
\begin{equation}\label{t1c}
\exists\,C>0,\ |\rho^1(\bm x)-\rho^2(\bm x)|\le C\,\mathrm{dist}(\bm x,\pa\Om)^2,\quad\bm x\in\Om.
\end{equation}
\item[{\rm(iii)}]$c^1=c^2$ and \eqref{t1b}--\eqref{t1c} hold simultaneously true.
\end{enumerate}
Let $u^j\ (j=1,2)$ be the solutions to \eqref{eq1} with $\Phi$ given by \eqref{Phi}, $\bm B=\bm0$ and $(\rho,a,c)=(\rho^j,a^j,c^j)$. Then the condition
\begin{equation}\label{t1d}
a^1\pa_\nu u^1=a^2\pa_\nu u^2\quad\mbox{on }(0,T_0)\times\Ga_\mathrm{out}
\end{equation}
implies $(\rho^1,a^1,c^1)=(\rho^2,a^2,c^2)$.
\end{theorem}

Second, we focus our attention on the simultaneous recovery of the convection term $\bm B$, the weight $\rho$ and the fractional power $\al$ under the assumption that $a=1$, $c=0$. Our second main result can be stated as follows.

\begin{theorem}\label{t2}
Let $d\ge3,\ \al^j\in(0,2),\ \rho^j\in L^\infty(\Om)$ satisfy \eqref{ell} and $\bm B^j\in(C^\ka(\ov\Om))^d$ with $\ka\in(2/3,1)\ (j=1,2)$. Let $u^j\ (j=1,2)$ be the solutions to \eqref{eq1} with  $a\equiv1,\ c\equiv0$, $(\al,\rho,\bm B)=(\al^j,\rho^j,\bm B^j)$ and for $\Phi$ given by \eqref{Phi} with $\chi\equiv1$. Then the condition
\[
\pa_\nu u^1=\pa_\nu u^2\quad\mbox{on }(0,T_0)\times\pa\Om
\]
implies $(\al^1,\rho^1,\bm B^1)=(\al^2,\rho^2,\bm B^2)$.
\end{theorem}

Now we turn to the simultaneous recovery of the manifold $(\cM,g)$ and the coefficients $(\mu,c)$ from a single measurement of the solution of \eqref{eqM1} on $\Ga_\out$ in some suitable sense.

{\bf Similarly to that in the Euclidean case, we select $\Ga'_\rin\subset\pa\cM$ and $\chi\in C^\infty(\pa\cM)$ such that $\supp\,\chi\subset\Ga_\rin$ and $\chi=1$ in $\Ga'_\rin$. Then we define the Dirichlet input $\Psi$ in \eqref{eqM1} exactly the same as \eqref{Phi}, where $\{\eta_k\}_{k\in\BN}$ is a sequence of $H^{3/2}(\pa\cM)$ such that $\rspan\{\eta_k\}$ is dense in $H^{3/2}(\pa\cM)$ and $\|\eta_k\|_{H^{3/2}(\pa\cM)}=1$ ($k\in\BN$).}

We will state three different extensions of Theorem \ref{t1}, among which the first one can be stated as follows.

\begin{corollary}\label{t4}
For $j=1,2$, let $(\cM^j,g^j)$ be two compact and smooth connected Riemannian manifolds of dimension $d\ge2$ with the same boundary, and let $\mu^j\in C^\infty(\cM^j)$ and $c^j\in C^\infty(\cM^j)$ satisfy $\mu^j>0$ and $c^j\ge0$ on $\cM^j$. Suppose
\[
\Ga_\rin=\Ga_\out\subset\pa\cM^1,\quad g^1=g^2,\ \mu^1=\mu^2=1,\ \pa_\nu\mu^1=\pa_\nu\mu^2=0\mbox{ on }\pa\cM^1.
\]
Denote by $u^j$ ($j=1,2$) the solution of \eqref{eqM1} with $\Phi$ given by \eqref{Phi} and $(\cM,g,\mu,c)=(\cM^j,g^j,\mu^j,c^j)$. Then the condition
\begin{equation}\label{t4d}
\pa_\nu u^1=\pa_\nu u^2\quad\mbox{in }(0,T_0)\times\Ga_\out
\end{equation}
implies that $(\cM^1,g^1)$ and $(\cM^2,g^2)$ are isometric. Moreover, \eqref{t4d} implies that there exist $\psi\in C^\infty(\cM^2;\cM^1)$ and $\ka\in C^\infty(\cM^2)$ satisfying
\begin{equation}\label{kappa}
\ka=1,\ \pa_\nu\ka=0\quad\mbox{on }\pa\cM^2
\end{equation}
such that
\begin{equation}\label{mu_c}
\mu^2=\ka^{-2}\mu^1\circ\psi,\quad c^2=c^1\circ\psi-\ka\tri_{g^2,\mu^1}\ka^{-1}.
\end{equation}
\end{corollary}

We can also extend our results to the case where the excitation and the measurements are disjoint. To this end, we need the following definition.

\begin{definition}
Consider the initial-boundary value problem for a hyperbolic equation with Dirichlet data $\Psi\in C_0^\infty((0,\infty)\times\pa\cM)$
\begin{equation}\label{eq_wave}
\begin{cases}
(\pa_t^2-\tri_g+c)u=0 & \mbox{in }(0,\infty)\times\cM,\\
u=\pa_t u=0 & \mbox{in }\{0\}\times\cM,\\
u=\Psi & \mbox{on }(0,\infty)\times\pa\cM,
\end{cases}
\end{equation}
where $\tri_g=\tri_{\mu,g}$ with $\mu=1$. We say that \eqref{eq_wave} is exactly controllable from a sub-boundary $\Ga\subset\pa\cM$ if there exists $T>0$ such that the map
\[
L^2((0,T)\times\Ga)\ni\Psi\longmapsto(u(T,\,\cdot\,),\pa_t u(T,\,\cdot\,))\in L^2(\cM)\times H^{-1}(\cM)
\]
is surjective.
\end{definition}

We refer to \cite{Bardos1992} for geometrical conditions that guarantee the exact controllability of \eqref{eq_wave}. We can now state the following two results with data on disjoint sets.

\begin{corollary}\label{t5}
Let $(\cM^j,g^j)$ ($j=1,2$) be two compact and smooth connected Riemannian manifolds of dimension $d\ge2$ with the same boundary. Denote by $u^j$ ($j=1,2$) the solution of \eqref{eqM1} with $\Phi$ given by \eqref{Phi}, $(\cM,g)=(\cM^j,g^j)$ ($j=1,2$) and $\mu=1,\ c=0$ on $\cM^1$. In addition, we assume that \eqref{eq_wave} is exactly controllable from $\Ga'_\rin$ and $\Ga_\rin\cap\Ga_\out=\emptyset$. Then \eqref{t4d} implies that $(\cM^1,g^1)$ and $(\cM^2,g^2)$ are isometric.
\end{corollary}

\begin{corollary}\label{t6}
Let $(\cM,g)$ be a compact and smooth connected Riemannian manifolds of dimension $d\ge2$. Let $\mu=1$ on $\cM$ and $c^j\in C^\infty(\cM)$ ($j=1,2$) be non-negative. We assume that \eqref{eq_wave} is exactly controllable from $\Ga'_\rin$, $\Ga_\out$ is strictly convex and $\Ga_\rin\cap\Ga_\out=\emptyset$. Then \eqref{t4d} implies that there exists a neighborhood $U$ of $\Ga_\out$ in $\cM$ such that $c^1=c^2$ in $U$.
\end{corollary}

\subsection{Comments about our results}

To the best of our knowledge, Theorems \ref{t1} and \ref{t2} are the first results on the recovery of coefficients appearing in fractional diffusion equation in a general bounded domain $\Om\subset\BR^d$ with $d\ge2$ from a single boundary measurement. For existing literature using other types of observation data with spatial dimensions $d\ge2$, we refer to \cite{LIY,KOSY,KSY} for an infinite number of boundary measurements, \cite{HLYZ} for interior observation, and \cite{KY1} for a single measurement on a cylindrical domain. Moreover, Theorem \ref{t2} seems to be the first result of unique recovery of a convection term appearing in a fractional diffusion equation.

Our approach is based on a delicate choice \eqref{Phi} of the boundary input $\Phi$ inspired by \cite{AS,CY} and suitable time-analyticity properties of the solutions to \eqref{eq1} and \eqref{eqM1}. More precisely, the key ingredient in our approach comes from the results stated in Propositions \ref{l1} and \ref{l2} stating that, for $k\in\BN$ and $\ve_k\in(0,(t_{2k}-t_{2k-1})/2)$, the restriction of the solutions of \eqref{eq2} and \eqref{eq3} to $(t_{2k-1}+\ve_k,\infty)$ are analytic in time as functions taking values in $H^{2\ga}(\Om)$ for some $\ga\in(3/4,1]$. For $\al=1$, this can be easily deduced from a classical lifting arguments and the time-analyticity of solutions to parabolic equations with time independent coefficients and source terms compactly supported in time. However, for $\al\ne1$, due to the presence of the nonlocal operator $\pa_t^\al$, this approach becomes more difficult. Instead, we use a new representation of the solution to \eqref{eq2} which, combined with some decay properties of the Mittag-Leffler functions, allows us to prove an analytic extension in time of the solution to \eqref{eq2} into a conical neighborhood of $(t_{2k-1}+\ve_k,\infty)$. For \eqref{eq3}, we employ Proposition \ref{l1} to built a sequence of analytic functions on a conical neighborhood of $(t_{2k-1}+\ve_k,\infty)$ taking values in $H^{2\ga}(\Om)$, and then show its convergence to an extension of the solution to \eqref{eq3}. Once these results are proved,  we complete the proof of Theorems \ref{t1} and \ref{t2} by transforming our inverse problem into an inverse boundary value problem stated for a family of elliptic equations. Our approach not only simplifies and extends the result of \cite{AS,CY} to fractional diffusion equation, but it also improves the result in \cite{AS} even for $\al=1$ since Theorem \ref{t1} is stated with partial boundary measurements and for more general type of coefficients.

In contrast to \cite{AS,CY}, the result of Theorem \ref{t1} is not based on recovering the coefficients appearing in elliptic equations from the associated Dirichlet-to-Neumann map. Instead, we prove Theorem \ref{t1} by applying some inverse spectral results of \cite{CK2}. Namely, in a similar way to \cite{KOSY}, we prove the recovery of the boundary spectral data associated with the elliptic operator appearing in \eqref{eq1}. This approach allows us to consider a general class of coefficients as well as  the sub-boundaries $\Ga_\rin,\Ga_\out$ subjected only to the conditions \eqref{t1a}. In the framework of a smooth Riemannian manifold and smooth coefficients, we even prove in Corollary \ref{t4} that the condition on $\Ga_\rin$ and $\Ga_\out$ can be reduced to the weaker condition $\Ga_\rin=\Ga_\out$  for the recovery of coefficients as well as the Riemannian manifold up to an isometry. In Corollaries \ref{t5} and \ref{t6}, we further consider the recovery of a manifold and some local recovery of coefficients when $\Ga_\rin\cap\Ga_\out=\emptyset$. The results of Corollaries \ref{t4}, \ref{t5} and \ref{t6} are based on the strategy of Theorem \ref{t1}, combined with the result of \cite{KKL,KKLO,LO} based on the boundary control method initiated by the work \cite{B,BK}.

Next, we observe that the results of Theorem \ref{t1} and Corollaries \ref{t4}, \ref{t5} and \ref{t6} are similar to that in \cite{KOSY} with different type of measurements. More precisely, \cite{KOSY} use an infinite number of measurements, whereas in Theorem \ref{t1} and Corollaries \ref{t4}, \ref{t5} and \ref{t6} we state our results with a single measurement. However, the measurements in \cite{KOSY} are taken at one fix time, while the measurements in Theorem \ref{t1} and Corollaries \ref{t4}, \ref{t5} and \ref{t6} are taken in a time interval $(0,T_0)$. In such a sense, Theorem \ref{t1} and Corollaries \ref{t4}, \ref{t5} and \ref{t6} can be regarded as a supplementation to that in \cite{KOSY} because the amounts of data information are essentially the same.

Finally, let us remark that by applying \cite{CY1,CY2}, we can extend without any difficulty Theorem \ref{t2} to the case $d=2$. However, since the result in that context will require a different definition of the input $\Phi$, we do not consider it in the present paper.

\section{Time-analyticity of solutions}\label{sec-analytic}

In this section, we investigate the analyticity of the solutions to \eqref{eq1} and \eqref{eqM1} in time. Without loss of generality, we only deal with the Euclidean case \eqref{eq1} and basically assume \eqref{ell} for the coefficients $(\rho,a,c)$. Due to the technical difference, we first treat the non-advection case $\bm B=\bm0$ of \eqref{eq1} in Subsection \ref{sec-analytic-B0}, and then proceed to the general situation in Subsection \ref{sec-analytic-general}.

\subsection{The case of $B=0$}\label{sec-analytic-B0}

Let $k\in\BN$ and consider the initial boundary value problem
\begin{equation}\label{eq2}
\begin{cases}
\rho\,\pa_t^\al u_k-\rdiv(a\nb u_k)+c\,u_k=0 & \mbox{in }\BR_+\times\Om,\\
\begin{cases}
u_k=0 & \mbox{if }0<\al\le1,\\
u_k=\pa_t u_k=0 & \mbox{if }1<\al<2
\end{cases} & \mbox{in }\{0\}\times\Om,\\
u_k=b_k\,\chi\,\psi_k\,\eta_k & \mbox{on }\BR_+\times\pa\Om.
\end{cases}
\end{equation}
We fix also $\ve_k\in(0,(t_{2k}-t_{2k-1})/2)$, $\te\in(0,\pi\min(\f1\al-\f12,\f12))$ and we fix $D_{k,\te}=\{t_{2k-1}+\ve_k+r\,\e^{\ri\be}\mid\be\in(-\te,\te)\}$. Then, we consider the following intermediate result.

\begin{proposition}\label{l1}
The solution of \eqref{eq2} restricted to $(t_{2k-1}+\ve_k,\infty)\times\Om$ can be extended uniquely to a function $\wt u_k\in C^1(\ov{D_{k,\te}}\,;H^2(\Om))\cap C^\om(D_{k,\te};H^2(\Om))$.
\end{proposition}

\begin{proof}
We denote by $\cA$ the operator defined by
\[
\cA f:=\f{-\rdiv(a\nb f)+c\,f}\rho,\quad f\in\cD(\cA)
\]
with domain $\cD(\cA)=\{f\in H^1_0(\Om)\mid\rdiv(a\nb f)\in L^2(\Om)\}$ acting on $L^2(\Om;\rho\,\rd\bm x)$. It is well known that such an operator is a self-adjoint operator with a spectrum consisting in the non-decreasing sequence of strictly positive eigenvalues $\{\la_\ell\}_{\ell\in\BN}$ and an associated orthonormal basis of eigenfunctions $\{\vp_\ell\}_{\ell\in\BN}$. Then, fixing $\ell\ge1$, taking the scalar product of \eqref{eq2} with $\vp_\ell$ and integrating by parts, we deduce that $u_{k,\ell}(t):=(u_k(t,\,\cdot\,),\rho\,\vp_\ell)$ solves the fractional ordinary differential equation
\[
\begin{cases}
\pa_t^\al u_{k,\ell}(t)+\la_\ell u_{k,\ell}(t)=-b_k\,\psi_k(t)\langle\chi\,\eta_k,a\,\pa_\nu\vp_\ell\rangle, & t>0,\\
\begin{cases}
u_{k,\ell}(0)=0 & \mbox{if }0<\al\le1,\\
u_{k,\ell}(0)=u'_{k,\ell}(0)=0 & \mbox{if }1<\al<2.
\end{cases}
\end{cases}
\]
It follows that
\[
u_{k,\ell}(t)=-b_k\langle\chi\,\eta_k,a\,\pa_\nu\vp_\ell\rangle\int_0^t(t-s)^{\al-1}E_{\al,\al}(-\la_\ell(t-s)^\al)\psi_k(s)\,\rd s.
\]
Using the fact that $\psi_k(0)=0$, integrating by parts and applying \cite[Lemma 3.2]{SY} yield
\[
u_{k,\ell}(t)=\f{b_k\langle\chi\,\eta_k,a\,\pa_\nu\vp_\ell\rangle}{\la_\ell}\left(-\psi_k(t)+\int_0^t E_{\al,1}(-\la_\ell(t-s)^\al)\psi_k'(s)\,\rd s\right).
\]
Setting
\begin{align}
v_k(t,\,\cdot\,) & :=-b_k\,\psi_k(t)\sum_{\ell=1}^\infty\f{\langle\chi\,\eta_k,a\,\pa_\nu\vp_\ell\rangle}{\la_\ell}\vp_\ell,\\
w_k(t,\,\cdot\,) & :=b_k\sum_{\ell=1}^\infty\f{\langle\chi\,\eta_k,a\,\pa_\nu\vp_\ell\rangle}{\la_\ell}\left(\int_0^t E_{\al,1}(-\la_\ell(t-s)^\al)\psi_k'(s)\,\rd s\right)\vp_\ell
\end{align}
for $t>0$, we see $u_k=v_k+w_k$. Therefore, it suffices to prove that the restriction of $v_k$, $w_k$ to $(t_{2k-1}+\ve_k,\infty)\times\Om$ can be extended to the functions $\wt v_k,\wt w_k\in C^1(\ov{D_{k,\te}}\,;H^2(\Om))\cap C^\om(D_{k,\te};H^2(\Om))$. For $v_k$, we claim that $v_k=b_k\,\psi_k\,G_k$, where $G_k$ solves
\[
\begin{cases}
-\rdiv(a\nb G_k)+c\,G_k=0 & \mbox{in }\Om,\\
G_k=\chi\,\eta_k & \mbox{on }\pa\Om.
\end{cases}
\]
To see this, it is enough to take the scalar product of $G_k$ with $\vp_\ell$ and integrate by parts to find
\begin{align*}
\la_\ell(G_k,\rho\,\vp_\ell) & =(G_k,-\rdiv(a\nb\vp_\ell)+c\,\vp_\ell)\\
& =-\langle\chi\,\eta_k,a\,\pa_\nu\vp_\ell\rangle+(-\rdiv(a\nb G_k)+c\,G_k,\vp_\ell)\\
& =-\langle\chi\,\eta_k,a\,\pa_\nu\vp_\ell\rangle.
\end{align*}
Thus, using the fact that $\Om$ is a $C^{1,1}$ domain and $\eta_k\in H^{3/2}(\pa\Om)$, we deduce $G_k\in H^2(\Om)$ and $\|G_k\|_{H^2(\Om)}\le C\|\eta_k\|_{H^{3/2}(\pa\Om)}$ (see e.g. \cite[Theorem 2.2.2.3]{Gr}). Moreover, from the definition of $\psi_k$, we deduce that
\[
v_k=b_k\,p_k\,G_k\quad\mbox{in }[t_{2k-1},\infty)\times\Om.
\]
This clearly proves that the restriction of $v_k$ to $(t_{2k-1}+\ve_k,\infty)\times\Om$ can be extended uniquely to $\wt v_k\in C^1(\ov{D_{k,\te}}\,;H^2(\Om))\cap C^\om(D_{k,\te};H^2(\Om))$.

Now let us consider $w_k$. Note first that thanks to \cite[Theorem 1.6]{P} as well as the facts that $\psi_k'=0$ on $(t_{2k-1},\infty)$ and
\begin{equation}\label{esti}
\sum_{\ell=1}^\infty\left|\f{\langle\chi\,\eta_k,a\,\pa_\nu\vp_\ell\rangle}{\la_\ell}\right|^2=\|G_k\|_{L^2(\Om)}^2<\infty,
\end{equation}
we have $w_k\in C([0,\infty);L^2(\Om))$ and
\[
w_k(t,\,\cdot\,):=b_k\sum_{\ell=1}^\infty\f{\langle\chi\,\eta_k,a\,\pa_\nu\vp_\ell\rangle}{\la_\ell}\left(\int_0^{t_{2k-1}}E_{\al,1}(-\la_\ell(t-s)^\al)\psi_k'(s)\,\rd s\right)\vp_\ell,\quad t\in(t_{2k-1},\infty).
\]
For any $\ell\ge1$, we introduce
\[
\begin{aligned}
& h_\ell(z):=\f{b_k\langle\chi\,\eta_k,a\,\pa_\nu\vp_\ell\rangle}{\la_\ell}\int_0^{t_{2k-1}}E_{\al,1}(-\la_\ell(z-s)^\al)\psi_k'(s)\,\rd s,\\
& w_{k,N}(z):=\sum_{\ell=1}^N h_\ell(z)\vp_\ell,
\end{aligned}
\quad z\in D_{k,\te}.
\]
Let $z_*\in\ov{D_{k,\te}}$ and $K$ be a compact neighborhood of $z_*$ with respect to the topology induced by $\ov{D_{k,\te}}\,$. It is clear that
\[
\rRe(z-s)\ge t_{2k-1}+\ve_k-t_{2k-1}=\ve_k,\quad z\in K,\ s\in[0,t_{2k-1}].
\]
Thus, for all $s\in[0,t_{2k-1}]$, the function $z\longmapsto(z-s)^\al$ is holomorphic in $K$ and we deduce that $w_{k,N}$ is an holomorphic function in $K$ taking values in $\cD(\cA)$. On the other hand, we have
\[
-(z-s)^\al\in\{r\,\e^{\ri\be}\mid r>\ve_k^\al,\ \be\in(\pi-\al\te,\pi+\al\te)\},\quad z\in K,\ s\in(0,t_{2k-1})
\]
and $\al\te\in(0,\pi-\f{\pi\al}2)$. Therefore, applying \cite[Theorem 1.6]{P}, we deduce that
\begin{align*}
|h_\ell(z)| & \le\f{|\langle\chi\,\eta_k,a\,\pa_\nu\vp_\ell\rangle|}{\la_\ell}\int_0^{t_{2k-1}}\f{C|\psi_k'(s)|}{1+\la_\ell|z-s|^\al}\,\rd s\le\f{|\langle\chi\,\eta_k,a\,\pa_\nu\vp_\ell\rangle|C\,t_{2k-1}\|\psi_k\|_{W^{1,\infty}(\BR)}}{\la_\ell(1+\ve_k^\al\la_\ell)}\\
& \le C_k\la_\ell^{-2}|\langle\chi\,\eta_k,a\,\pa_\nu\vp_\ell\rangle|,\quad z\in K,\ \ell\ge1,
\end{align*}
where $C_k$ is a constant independent of $z$ and $\ell$. Thus, applying \eqref{esti} yields
\[
\sup_{z\in K}\sum_{\ell=1}^\infty|\la_\ell h_\ell(z)|^2\le C_k\sum_{\ell=1}^\infty\left|\f{\langle\chi\,\eta_k,a\,\pa_\nu\vp_\ell\rangle}{\la_\ell}\right|^2<\infty.
\]
This proves that $w_{k,N}$ converge to the function
\[
\wt w_k(z)=\sum_{\ell=1}^\infty h_\ell(z)\vp_\ell
\]
uniformly with respect to $z\in K$ as a function taking values in $\cD(\cA)$. Using the fact $z_*$ is arbitrarily chosen in $\ov{D_{k,\te}}\,$, we deduce that $\wt w_k\in C^1(\ov{D_{k,\te}}\,;\cD(\cA))\cap C^\om(D_{k,\te};\cD(\cA))$. On the other hand, since $\Om$ is $C^{1,1}$, by the regularity of the operator $\cA$ (see \cite[Theorem 2.2.2.3]{Gr}), we deduce that $\cD(\cA)$ is embedded continuously into $H^2(\Om)$ and $\wt w_k\in C^1(\ov{D_{k,\te}}\,;H^2(\Om))\cap C^\om(D_{k,\te};H^2(\Om))$. Combining this with the fact that
\[
w_k(t)=\wt w_k(t),\quad t\in(t_{2k-1}+\ve_k,\infty),
\]
we deduce that $\wt u_k=\wt v_k+\wt w_k\in C^1(\ov{D_{k,\te}}\,;H^2(\Om))\cap C^\om(D_{k,\te};H^2(\Om))$ is the unique holomorphic extension of $u_k$ restricted to $(t_{2k-1}+\ve_k,\infty)$.
\end{proof}

\subsection{The case of $B\ne0$}\label{sec-analytic-general}

In this subsection, we restrict $a=1$, $c=0$ in \eqref{eq1} and assume $(\rho,\bm B)\in(L^\infty(\Om))^{d+1}$ with $\rho$ satisfying \eqref{ell}. Parallel to the previous subsection, let $k\in\BN$ and consider the initial-boundary value problem
\begin{equation}\label{eq3}
\begin{cases}
\rho\,\pa_t^\al v_k-\tri v_k=-\bm B\cdot\nb v_k & \mbox{in }\BR_+\times\Om,\\
\begin{cases}
v_k=0 & \mbox{if }0<\al\le1,\\
v_k=\pa_t v_k=0 & \mbox{if }1<\al<2
\end{cases} & \mbox{in }\{0\}\times\Om,\\
v_k=b_k\,\chi\,\psi_k\,\eta_k & \mbox{on }\BR_+\times\pa\Om.
\end{cases}
\end{equation}
Regarding the advection term as a new source, we see that the equation \eqref{eq3} admits a unique solution $v_k\in C([0,\infty);H^{2\ga}(\Om))$ ($s\in(3/4,1)$) taking the form of
\[
v_k(t,\,\cdot\,)=u_k(t,\,\cdot\,)+\sum_{\ell=1}^\infty\int_0^t(t-s)^{\al-1}E_{\al,\al}(-\la_\ell(t-s)^\al)(\bm B\cdot\nb v_k(s,\,\cdot\,),\vp_\ell)\,\rd s\,\vp_\ell,
\]
where $u_k$ is the solution of \eqref{eq2}. Moreover, for all $T>0$, one can prove the following estimate
\[
\|v_k(t,\,\cdot\,)\|_{H^{2\ga}(\Om)}\le C\|\eta_k\|_{H^{3/2}(\pa\Om)},\quad0<t\le T.
\]

We will show that the restriction of $v_k$ to $(t_{2k-1}+\ve_k,\infty)\times\Om$ admits an holomorphic extension as a function taking its value in $H^{2\ga}(\Om)$.

\begin{proposition}\label{l2}
Let $s\in(3/4,1)$. Then the solution of \eqref{eq3} restricted to $(t_{2k-1}+\ve_k,\infty)\times\Om$ can be extended to $\wt v_k\in C^1(\ov{D_{k,\te}}\,;H^{2\ga}(\Om))\cap C^\om(D_{k,\te};H^{2\ga}(\Om))$.
\end{proposition}

\begin{proof}
Let us fix $v_k^0=0$ and, for $n\in\BN$ and $z\in\ov{D'_{k,\te}}\,$, define
\begin{align}
v_k^n(z,\,\cdot\,) & :=\wt u_k(z,\,\cdot\,)+\wt y_k(z,\,\cdot\,)\nonumber\\
& \quad\,\:+\sum_{\ell=1}^\infty\left(\int_0^{z-t_{2k-1}-\ve_k}\ze^{\al-1}E_{\al,\al}(-\la_\ell\ze^\al)\left(\bm B\cdot\nb v_k^{n-1}(z-\ze,\,\cdot\,),\vp_\ell\right)\rd\ze\right)\vp_\ell,\label{vkn}
\end{align}
where
\[
\wt y_k(z,\,\cdot\,)=\sum_{\ell=1}^\infty\left(\int_0^{t_{2k-1}+\ve_k}(z-\ze)^{\al-1}E_{\al,\al}(-\la_\ell(z-\ze)^\al)(\bm B\cdot\nb v_k(\ze,\,\cdot\,),\vp_\ell)\,\rd\ze\right)\vp_\ell.
\]
We divide the proof into three steps. We start by proving that $v_k^n\in C(\ov{D_{k,\te}}\,;H^{2\ga}(\Om))\cap C^\om(D_{k,\te};H^{2\ga}(\Om))$ for all $n\in\BN$. Since we have $\wt u_k\in C^1(\ov{D_{k,\te}}\,;H^{2\ga}(\Om))\cap C^\om(D_{k,\te};H^{2\ga}(\Om))$ in view of Proposition \ref{l1}, then it suffices to show that $\wt y_k\in C(\ov{D_{k,\te}}\,;H^{2\ga}(\Om))\cap C^\om(D_{k,\te};H^{2\ga}(\Om))$. Once this is proved, we can complete the proof by showing that $\{v_k^n\}_{n\in\BN}$ converges uniformly to $\wt v_k$ on any compact subset of $\ov{D_{k,\te}}$, which coincides on $(t_{2k-1}+\ve_k,\infty)$ with $v_k$.\medskip

{\bf Step 1}\ \ We prove that $\wt y_k\in C(\ov{D_{k,\te}}\,;H^{2\ga}(\Om))\cap C^\om(D_{k,\te};H^{2\ga}(\Om))$. For this purpose, we fix $\de\in(0,t_{2k-1}+\ve_k)$ and consider
\[
\wt y_{k,\de}(z,\,\cdot\,)=\sum_{\ell=1}^\infty\left(\int_0^{t_{2k-1}+\ve_k-\de}(z-\ze)^{\al-1}E_{\al,\al}(-\la_\ell(z-\ze)^\al)(\bm B\cdot\nb v_k(s,\,\cdot\,),\vp_\ell)\,\rd\ze\right)\vp_\ell,\quad z\in\ov{D_{k,\te}}\,.
\]
Repeating the arguments used in the proof of Proposition \ref{l1}, one can check that $\wt y_{k,\de}\in C^1(\ov{D_{k,\te}}\,;H^{2\ga}(\Om))\cap C^\om(D_{k,\te};H^{2\ga}(\Om))$. Moreover, for any compact set $K\subset\ov{D_{k,\te}}\,$, applying \cite[Theorem 1.6]{P}, we obtain
\begin{align*}
\|(\wt y_k-\wt y_{k,\de})(z,\,\cdot\,)\|_{H^{2\ga}(\Om)} & \le C\|(\wt y_k-\wt y_{k,\de})(z,\,\cdot\,)\|_{\cD(\cA^\ga)}\\
& \le C\|v_k\|_{L^\infty(0,T;H^1(\Om))}\left(\int_{t_{2k-1}+\ve_k-\de}^{t_{2k-1}+\ve_k}\ze^{\al(1-\ga)-1}\,\rd\ze\right)\\
& \le C\left\{(t_{2k-1}+\ve_k)^{\al(1-\ga)}-(t_{2k-1}+\ve_k-\de)^{\al(1-\ga)}\right\},\quad z\in K,
\end{align*}
where $C>0$ is a constant independent of $z$. This proves that $\wt y_{k,\de}$ converges to $\wt y_k$ uniformly for $z\in K$ in $H^{2\ga}(\Om)$ as $\de\to0$. From this result, we deduce that $\wt y_k\in C^1(\ov{D_{k,\te}}\,;H^{2\ga}(\Om))\cap C^\om(D_{k,\te};H^{2\ga}(\Om))$.

{\bf Step 2}\ \ We show by induction that $v_k^n\in C^1(\ov{D_{k,\te}}\,;H^{2\ga}(\Om))\cap C^\om(D_{k,\te};H^{2\ga}(\Om))$ for all $n\in\BN$.

It is clear that this property is true for $n=0$. Now suppose that for some $n\in\BN$, there holds $v_k^j\in C^1(\ov{D_{k,\te}}\,;H^{2\ga}(\Om))\cap C^\om(D_{k,\te};H^{2\ga}(\Om))$ for $j=0,\ldots,n-1$. Consider $v_k^n$ defined on $z\in\ov{D_{k,\te}}$ by \eqref{vkn} and
\[
w_k^n(z):=\sum_{\ell=1}^\infty\left(\int_0^{z-t_{2k-1}-\ve_k}\ze^{\al-1}E_{\al,\al}(-\la_\ell\ze^\al)\left(\bm B\cdot\nb v_k^{n-1}(z-\ze,\,\cdot\,),\vp_\ell\right)\rd\ze\right)\vp_\ell.
\]
According to the above discussion, one can complete the proof by showing that $w_k^n\in C^1(\ov{D_{k,\te}}\,;H^{2\ga}(\Om))\cap C^\om(D_{k,\te};H^{2\ga}(\Om))$. To this end, let us consider
\[
r_\ell^n(z):=\int_0^{z-t_{2k-1}-\ve_k}\ze^{\al-1}E_{\al,\al}(-\la_\ell\ze^\al)\left(\bm B\cdot\nb v_k^{n-1}(z-\ze,\,\cdot\,),\vp_\ell\right)\rd\ze,\quad\ell\in\BN.
\]
Since $v_k^{n-1}\in C^1(\ov{D_{k,\te}}\,;H^{2\ga}(\Om))$, one can easily check that $r_\ell^n\in C^1(\ov{D_{k,\te}})$. Moreover, for any $z\in D_{k,\te}$ and $\tau\in\BC$ such that $|\tau|$ is sufficiently small, we have
\[
\f{r_\ell^n(z+\tau)-r_\ell^n(z)}\tau=:\I+\I\I,
\]
where
\begin{align*}
\I & :=\f1\tau\int_0^{z+\tau-t_{2k-1}-\ve_k}\ze^{\al-1}E_{\al,\al}(-\la_\ell\ze^\al)\left(\bm B\cdot\nb v_k^{n-1}(z-\ze,\,\cdot\,),\vp_\ell\right)\rd\ze\\
& \quad\,\:-\f1\tau\int_0^{z-t_{2k-1}-\ve_k}\ze^{\al-1}E_{\al,\al}(-\la_\ell\ze^\al)\left(\bm B\cdot\nb v_k^{n-1}(z-\ze,\,\cdot\,),\vp_\ell\right)\rd\ze\\
\I\I & :=\int_0^{z+\tau-t_{2k-1}-\ve_k}\ze^{\al-1}E_{\al,\al}(-\la_\ell\ze^\al)\left(\bm B\cdot\nb\left(\f{v_k^{n-1}(z+\tau-\ze,\,\cdot\,)-v_k^{n-1}(z-\ze,\,\cdot\,)}\tau\right),\vp_\ell\right)\rd\ze.
\end{align*}
Fix $\de\in(0,\rRe\,z-t_{2k-1}-\ve_k)$. Since $v_k^{n-1}\in C^\om(D_{k,\te};H^{2\ga}(\Om))$, one can find a neighborhood $\cO$ of $[\de,z-t_{2k-1}-\ve_k]\cup[\de,z+\tau-t_{2k-1}-\ve_k]$ such that
\[
\ze\longmapsto\ze^{\al-1}E_{\al,\al}(-\la_\ell\ze^\al)(\bm B\cdot\nb v_k^{n-1}(z-\ze,\,\cdot\,),\vp_\ell)\in C^\om(\cO;\BC).
\]
Therefore, we have
\begin{align*}
& \quad\,\int_\de^{z+\tau-t_{2k-1}-\ve_k}\ze^{\al-1}E_{\al,\al}(-\la_\ell\ze^\al)\left(\bm B\cdot\nb v_k^{n-1}(z-\ze,\,\cdot\,),\vp_\ell\right)\rd\ze\\
& \quad\,-\int_\de^{z-t_{2k-1}-\ve_k}\ze^{\al-1}E_{\al,\al}(-\la_\ell\ze^\al)\left(\bm B\cdot\nb v_k^{n-1}(z-\ze,\,\cdot\,),\vp_\ell\right)\rd\ze\\
& =\int_{z-t_{2k-1}-\ve_k}^{z+\tau-t_{2k-1}-\ve_k}\ze^{\al-1}E_{\al,\al}(-\la_\ell\ze^\al)\left(\bm B\cdot\nb v_k^{n-1}(z-\ze,\,\cdot\,),\vp_\ell\right)\rd\ze.
\end{align*}
Using the fact that $v_k^{n-1}\in C^1(\ov{D_{k,\te}}\,;H^{2\ga}(\Om))$, we can send $\de\to0$ and deduce
\begin{equation}\label{l2f}
\I=\f1\tau\int_{z-t_{2k-1}-\ve_k}^{z+\tau-t_{2k-1}-\ve_k}\ze^{\al-1}E_{\al,\al}(-\la_\ell\ze^\al)\left(\bm B\cdot\nb v_k^{n-1}(z-\ze,\,\cdot\,),\vp_\ell\right)\rd\ze.
\end{equation}
Combining this with the fact that $v_k^{n-1}\in C(\ov{D_{k,\te}}\,;H^{2\ga}(\Om))$, we deduce that
\[
\lim_{\tau\to0}\I=z^{\al-1}E_{\al,\al}(-\la_\ell z^\al)\left(\bm B\cdot\nb v_k^{n-1}(t_{2k-1}+\ve_k,\,\cdot\,),\vp_\ell\right).
\]
In the same way, for $\de>0$ small enough, using the fact that $v_k^{n-1}\in C^\om(D_{k,\te};H^{2\ga}(\Om))$, we have
\begin{align*}
& \quad\,\lim_{\tau\to0}\int_\de^{z+\tau-t_{2k-1}-\ve_k}\ze^{\al-1}E_{\al,\al}(-\la_\ell\ze^\al)\left(\bm B\cdot\nb\left(\f{v_k^{n-1}(z+\tau-\ze,\,\cdot\,)-v_k^{n-1}(z-\ze,\,\cdot\,)}\tau\right),\vp_\ell\right)\rd\ze\\
& =\int_\de^{z-t_{2k-1}-\ve_k}\ze^{\al-1}E_{\al,\al}(-\la_\ell\ze^\al)\left(\bm B\cdot\nb\pa_z v_k^{n-1}(z-\ze,\,\cdot\,),\vp_\ell\right)\rd\ze.
\end{align*}
Combining this with the fact that $v_k^{n-1}\in C^1(\ov{D_{k,\te}}\,;H^{2\ga}(\Om))$ and sending $\de\to0$, we obtain
\[
\lim_{\tau\to0}\I\I=\int_0^{z-t_{2k-1}-\ve_k}\ze^{\al-1}E_{\al,\al}(-\la_\ell\ze^\al)\left(\bm B\cdot\nb\pa_z v_k^{n-1}(z-\ze,\,\cdot\,),\vp_\ell\right)\rd\ze.
\]
Combining this with \eqref{l2f}, we deduce that $r_\ell^n\in C^1(\ov{D_{k,\te}})\cap C^\om(D_{k,\te};\BC)$. In addition, for any $m_1,m_2\in\BN$ satisfying $m_1<m_2$ and $K$ a compact subset of $D_{k,\te}$, applying \cite[Theorem 1.6]{P}, we obtain
\begin{align*}
\left\|\sum_{\ell=m_1}^{m_2}r_\ell^n(z)\vp_\ell\right\|_{H^{2\ga}(\Om)}^2 & \le C\left\|\sum_{\ell=m_1}^{m_2}r_\ell^n(z)\vp_\ell\right\|_{\cD(\cA^\ga)}^2\\
& \le C\sup_{z\in K_1}\sum_{\ell=m_1}^{m_2}\left|\left\langle\bm B\cdot\nb v_k^{n-1}(z,\,\cdot\,),\vp_\ell\right\rangle\right|^2\left(\int_0^1s^{\al(1-\ga)-1}\,\rd s\right)^2\\
& \le C\sup_{z\in K_1}\sum_{\ell=m_1}^{m_2}\left|\left\langle\bm B\cdot\nb v_k^{n-1}(z,\,\cdot\,),\vp_\ell\right\rangle\right|^2,
\end{align*}
where
\[
K_1:=\{z-\ze\mid z\in K,\ \ze\in[0,z-t_{2k-1}-\ve_k]\}
\]
is a compact subset of $\ov{D_{k,\te}}\,$. Combining this with the fact that $v_k^{n-1}\in C^1(\ov{D_{k,\te}}\,;H^{2\ga}(\Om))$, we deduce that the series $\sum_{\ell=1}^\infty r_\ell^n(z)\vp_\ell$ converges uniformly with respect to $z\in K$ to $w_k^n$ as functions taking values in $H^{2\ga}(\Om)$. This proves that $w_k^n\in C^1(\ov{D_{k,\te}}\,;H^{2\ga}(\Om))\cap C^\om(D_{k,\te};H^{2\ga}(\Om))$ and by the same way that $v_k^n\in C^1(\ov{D_{k,\te}}\,;H^{2\ga}(\Om))\cap C^\om(D_{k,\te};H^{2\ga}(\Om))$. By induction, it follows that this property holds true for all $n\in\BN$.\medskip

{\bf Step 3}\ \ Now we complete the proof by showing that the sequence $\{v_k^n\}_{n\in\BN}$ converges uniformly on any compact set $K\subset\ov{D_{k,\te}}$ to $\wt v_k\in C^1(\ov{D_{k,\te}}\,;H^{2\ga}(\Om))\cap C^\om(D_{k,\te};H^{2\ga}(\Om))$, which coincides with $v_k$ restricted to $(t_{2k-1}+\ve_k,\infty)$. To see this, let us first remark that $v_k^n$ can be rewritten as
\begin{align*}
v_k^n(z,\,\cdot\,) & =\wt u_k(z,\,\cdot\,)+\wt y_k(z,\,\cdot\,)\\
& \quad\,+\sum_{\ell=1}^\infty\left(\int_{t_{2k-1}+\ve_k}^z(z-\ze)^{\al-1}E_{\al,\al}(-\la_\ell(z-\ze)^\al)\left(\bm B\cdot\nb v_k^{n-1}(\ze,\,\cdot\,),\vp_\ell\right)\rd\ze\right)\vp_\ell,\quad z\in\ov{D_{k,\te}}\,.
\end{align*}
Therefore, combining \cite[Theorem 1.6]{P} with the Lebesgue dominate convergence theorem, we see that
\begin{align*}
(v_k^{n+1}-v_k^n)(z,\,\cdot\,)=\int_{t_{2k-1}+\ve_k}^z(z-\ze)^{\al-1}\left(\sum_{\ell=1}^\infty E_{\al,\al}(-\la_\ell(z-\ze)^\al)\left(\bm B\cdot\nb\left(v_k^n-v_k^{n-1}\right)(\ze,\,\cdot\,),\vp_\ell\right)\vp_\ell\right)\rd\ze.
\end{align*}
Using this identity, we can obtain the following estimate of $\|v_k^{n+1}-v_k^n\|_{\cD(\cA^\ga)}$ by an inductive argument.

\begin{lemma}\label{l3}
For any $n\in\BN$ and any compact subset $K$ of $\ov{D_{k,\te}}\,$, we have
\begin{equation}\label{iteration-esti}
\left\|\left(v_k^{n+1}-v_k^n\right)(z,\,\cdot\,)\right\|_{H^{2\ga}(\Om)}\le C_K\f{C^n\|\bm B\|_{(L^\infty(\Om))^d}^n|z-t_{2k-1}-\ve_k|^{\al n(1-\ga)}}{\Ga(\al n(1-\ga)+1)},\quad z\in K.
\end{equation}
\end{lemma}

For the sake of consistency, we provide the proof of Lemma \ref{l3} after finishing the proof of Proposition \ref{l2}. In view of Lemma \ref{l3}, it follows that the sequence $\{v_k^n\}_{n=0}^\infty$ converges uniformly on any compact set $K\subset\ov{D_{k,\te}}$ to $\wt v_k\in C^1(\ov{D_{k,\te}}\,;H^{2\ga}(\Om))\cap C^\om(D_{k,\te};H^{2\ga}(\Om))$. Moreover, in light of \eqref{vkn}, we have
\[
\wt v_k(z,\,\cdot\,)=\wt u_k(z,\,\cdot\,)+\wt y_k(z,\,\cdot\,)+\sum_{\ell=1}^\infty\left(\int_{t_{2k-1}+\ve_k}^z(z-\ze)^{\al-1}E_{\al,\al}(-\la_\ell(z-\ze)^\al)(\bm B\cdot\nb\wt v_k(\ze,\,\cdot\,),\vp_\ell)\,\rd\ze\right)\vp_\ell
\]
for $z\in\ov{D_{k,\te}}\,$. In particular, we find
\[
\wt v_k(t,\,\cdot\,)=u_k(t,\,\cdot\,)+\wt y_k(t,\,\cdot\,)+\sum_{\ell=1}^\infty\left(\int_{t_{2k-1}+\ve_k}^t(t-s)^{\al-1}E_{\al,\al}(-\la_\ell(t-s)^\al)(\bm B\cdot\nb\wt v_k(s,\,\cdot\,),\vp_\ell)\,\rd s\right)\vp_\ell
\]
for $t>t_{2k-1}+\ve_k$. Combining this with the fact that
\[
v_k(t,\,\cdot\,)=u_k(t,\,\cdot\,)+\wt y_k(t,\,\cdot\,)+\sum_{\ell=1}^\infty\left(\int_{t_{2k-1}+\ve_k}^t(t-s)^{\al-1}E_{\al,\al}(-\la_\ell(t-s)^\al)(\bm B\cdot\nb v_k(s,\,\cdot\,),\vp_\ell)\,\rd s\right)\vp_\ell
\]
for $t>t_{2k-1}+\ve_k$ and the uniqueness of the integral equation
\[
w(t,\,\cdot\,)=u_k(t,\,\cdot\,)+\wt y_k(t,\,\cdot\,)+\sum_{\ell=1}^\infty\left(\int_{t_{2k-1}+\ve_k}^t(t-s)^{\al-1}E_{\al,\al}(-\la_\ell(t-s)^\al)(\bm B\cdot\nb w(s,\,\cdot\,),\vp_\ell)\,\rd s\right)\vp_\ell
\]
for $t>t_{2k-1}+\ve_k$, which can be deduced from arguments similar to \cite[Proposition 1]{FK}, we conclude that
\[
\wt v_k(t,\,\cdot\,)=v_k(t,\,\cdot\,),\quad t>t_{2k-1}+\ve_k.
\]
This completes the proof of Proposition \ref{l2}.
\end{proof}

Now we provide the proof of Lemma \ref{l3}.

\begin{proof}[Proof of Lemma \ref{l3}]
Let $z\in D_{k,\te}$. Using the fact that $z-t_{2k-1}-\ve_k=|z-t_{2k-1}-\ve_k|\,\e^{\ri\be}:=r_k\e^{\ri\be}$ with $\be\in(-\te,\te)$, we find
\begin{align*}
& \quad\,\left\|\left(v_k^{n+1}-v_k^n\right)(z,\,\cdot\,)\right\|_{\cD(\cA^\ga)}\\
& \le\left|\int_{t_{2k-1}+\ve_k}^z|z-\ze|^{\al-1}\left(\sum_{\ell=1}^\infty\la_\ell^{2\ga}\left|E_{\al,\al}(-\la_\ell(z-\ze)^\al)\left(\bm B\cdot\nb\left(v_k^n-v_k^{n-1}\right)(\ze,\,\cdot\,),\vp_\ell\right)\right|^2\right)^{1/2}|\rd\ze|\right|\\
& \le C\int_0^{r_k}(r_k-\tau)^{\al-1}\\
& \quad\,\times\left(\sum_{\ell=1}^\infty\la_\ell^{2\ga}\left|E_{\al,\al}(-\la_\ell(r_k-\tau)^\al\e^{\ri\al\be})\left(\bm B\cdot\nb\left(v_k^n-v_k^{n-1}\right)(t_{2k-1}+\ve_k+\tau\,\e^{\ri\be},\,\cdot\,),\vp_\ell\right)\right|^2\right)^{1/2}\,\rd\tau.
\end{align*}
Therefore, applying \cite[Theorem 1.6]{P}, we get
\begin{align*}
\left\|\left(v_k^{n+1}-v_k^n\right)(z,\,\cdot\,)\right\|_{\cD(\cA^\ga)} & \le C\|\bm B\|_{(L^\infty(\Om))^d}\int_0^{r_k}\f{(r_k-\tau)^{\al(1-\ga)-1}}{\Ga(\al(1-\ga))}\\
& \quad\,\times\left\|\left(v_k^n-v_k^{n-1}\right)(t_{2k-1}+\ve_k+\tau\,\e^{\ri\be},\,\cdot\,)\right\|_{H^{2\ga}(\Om)}\,\rd\tau\\
& \le C\|\bm B\|_{(L^\infty(\Om))^d}\int_0^{|z-t_{2k-1}-\ve_k|}\f{(|z-t_{2k-1}-\ve_k|-\tau)^{\al(1-\ga)-1}}{\Ga(\al(1-\ga))}\\
& \quad\,\times\left\|(v_k^n-v_k^{n-1})(t_{2k-1}+\ve_k+\tau\,\e^{\ri\be},\,\cdot\,)\right\|_{H^{2\ga}(\Om)}\,\rd\tau
\end{align*}
for $z\in D_{k,\te}$. Using this estimate and applying some arguments similarly to the proof of \cite[Proposition 1]{FK}, we can prove by iteration that
\begin{align*}
\|v_k^{n+1}(z,\,\cdot\,)-v_k^n(z,\,\cdot\,)\|_{H^{2\ga}(\Om)} & \le C^n\|\bm B\|_{(L^\infty(\Om))^d}^n\int_0^{|z-t_{2k-1}-\ve_k|}\f{(|z-t_{2k-1}-\ve_k|-\tau)^{n\al(1-\ga)-1}}{\Ga(n\al(1-\ga))}\\
& \quad\,\times\left\|(v_k^1-v_k^0)(t_{2k-1}+\ve_k+\tau\,\e^{\ri\be},\,\cdot\,)\right\|_{H^{2\ga}(\Om)}\,\rd\tau\\
& \le C^n\|\bm B\|_{(L^\infty(\Om))^d}^n\int_0^{|z-t_{2k-1}-\ve_k|}\f{(|z-t_{2k-1}-\ve_k|-\tau)^{n\al(1-\ga)-1}}{\Ga(n\al(1-\ga))}\\
& \quad\,\times\left\|(\wt u_k+\wt y_k)(t_{2k-1}+\ve_k+\tau\,\e^{\ri\be},\,\cdot\,)\right\|_{H^{2\ga}(\Om)}\,\rd\tau.
\end{align*}
Combining this last estimate with the fact that $\wt u_k,\wt y_k\in C^1(\ov{D_{k,\te}}\,;H^{2\ga}(\Om))$, we deduce \eqref{iteration-esti}.
\end{proof}

\section{Proofs of of the main results}\label{sec-proof}

\begin{proof}[Proof of Theorem \ref{t1}]
We dived the proof into four steps.\smallskip

{\bf Step 1}\ \ We start by proving that \eqref{t1d} implies
\begin{equation}\label{t1e}
a^1\pa_\nu u_k^1=a^2\pa_\nu u_k^2\quad\mbox{on }\BR_+\times\Ga_\out,\ k\in\BN,
\end{equation}
where $u_k^j$ ($j=1,2$, $k\in\BN$) is the solution of \eqref{eq2} with $(\rho,a,c)=(\rho^j,a^j,c^j)$ and $j=1,2$. We will prove \eqref{t1e} by induction. For $k=1$, using the properties of the sequence $\{\psi_k\}_{k\in\BN}$, we observe that
\[
\psi_k=0\quad\mbox{in }(0,t_2),\ \forall\,k\ge2.
\]
Therefore, we have $u_1^j=u^j$ in $(0,t_2)\times\Om$ ($j=1,2$), and the condition \eqref{t1d} implies
\begin{equation}\label{t1f}
a^1\pa_\nu u_1^1=a^2\pa_\nu u_1^2\quad\mbox{on }(0,t_2)\times\Ga_\out.
\end{equation}
On the other hand, from Proposition \ref{l1} we know that $u_1^j\in C^\om((t_1+\ve_1,\infty);H^2(\Om))$, $j=1,2$. Thus, it follows from the trace theorem that $t\longmapsto\pa_\nu u_1^j(t,\,\cdot\,)|_{\Ga_\out}\in C^\om((t_1+\ve_1,\infty);L^2(\Ga_\out))$ for $j=1,2$, and the condition \eqref{t1f} implies \eqref{t1e} for $k=1$. Now assume that for some $\ell\in\BN$, the condition \eqref{t1e} is fulfilled for all $k=1,2,\ldots,\ell$. Since
\[
\psi_k=0\quad\mbox{in }(0,t_{2\ell+2}),\ \forall\,k\ge\ell+2,
\]
we know that
\[
\sum_{k=1}^{\ell+1}u_k^j=u^j\quad\mbox{in }(0,t_{2\ell+2})\times\Om.
\]
Therefore, \eqref{t1d} implies
\[
\sum_{k=1}^{\ell+1}a^1\pa_\nu u_k^1=\sum_{k=1}^{\ell+1}a^2\pa_\nu u_k^2\quad\mbox{in }(0,t_{2\ell+2})\times\Ga_\out.
\]
Then, by the inductive assumption, we deduce that
\[
a^1\pa_\nu u_{\ell+1}^1=a^2\pa_\nu u_{\ell+1}^2\quad\mbox{in }(0,t_{2\ell+2})\times\Ga_\out.
\]
Applying again Proposition \ref{l1}, we deduce that $t\longmapsto\pa_\nu u_{\ell+1}^j(t,\,\cdot\,)|_{\Ga_\out}\in C^\om((t_{2\ell+1}+\ve_\ell,\infty);L^2(\Ga_\out))$ for $j=1,2$, and we conclude \eqref{t1e} for $k=\ell+1$. This proves that \eqref{t1e} holds for all $k\in\BN$.\smallskip

{\bf Step 2}\ \ For $j=1,2$ and $\xi>0$, consider the boundary value problem
\begin{equation}\label{t1h}
\begin{cases}
-\rdiv(a^j\nb U^j(\xi))+(\xi^\al\rho^j+c^j)U^j(\xi)=0 & \mbox{in }\Om,\\
U^j(\xi)=\chi\, h & \mbox{on }\pa\Om,
\end{cases}\quad h\in H^{3/2}(\pa\Om).
\end{equation}
We associate this problem with the Dirichlet-to-Neumann map
\[
\La^j(\xi):h\longmapsto\left.a^j\pa_\nu U^j(\xi)\right|_{\Ga_\out},\quad j=1,2,\ \xi>0.
\]
In this step, we show that \eqref{t1d} implies
\begin{equation}\label{t1i}
\La^1(\xi)=\La^2(\xi),\quad j=1,2,\ \xi>0.
\end{equation}
For $j=1,2$, define the operator $\cA^j$ by
\[
\cA^j w:=\f{-\rdiv(a^j\nb w)+c^j w}{\rho^j},\quad w\in\cD(\cA^j)
\]
with the domain $\cD(\cA^j)=\{w\in H^1_0(\Om)\mid\rdiv(a^j\nb w)\in L^2(\Om)\}$ acting on $L^2(\Om;\rho^j\,\rd\bm x)$. Then we associate these operators with the eigensystems $\{(\la_\ell^j,\vp_\ell^j)\}_{\ell\in\BN}$. Similarly to the proof of Proposition \ref{l1}, for all $k\in\BN$, $u^j_k$ can be decomposed into $u^j_k=v^j_k+w^j_k$ with $v^j_k=-b_k\,\psi_k\,G^j_k$, where $G^j_k$ and $w^j_k$ solve
\[
\begin{cases}
-\rdiv\left(a^j\nb G^j_k\right)+c^j G^j_k=0 & \mbox{in }\Om,\\
G^j_k=-\chi\,\eta_k & \mbox{on }\pa\Om
\end{cases}
\]
and
\[
\begin{cases}
\rho^j\pa_t^\al w^j_k-\rdiv\left(a^j\nb w^j_k\right)+c^j w^j_k=b_k(\pa_t^\al\psi_k)G^j_k & \mbox{in }\BR_+\times\Om,\\
\begin{cases}
w^j_k=0 & \mbox{if }0<\al\le1,\\
w^j_k=\pa_t w^j_k=0 & \mbox{if }1\le\al\le2
\end{cases} & \mbox{in }\{0\}\times\Om,\\
w^j_k=0 & \mbox{on }\BR_+\times\pa\Om,
\end{cases}
\]
respectively. For $\ga\in(3/4,1)$, using \cite[Theorem 1.6]{P}, we deduce that
\begin{align*}
\|w^j_k(t,\,\cdot\,)\|_{H^{2\ga}(\Om)}^2 & \le C\|w^j_k(t,\,\cdot\,)\|_{\cD((\cA^j)^\ga)}^2\\
& \le C|b_k|^2\sum_{\ell=1}^\infty\left|\int_0^t\left(\la^j_\ell\right)^\ga(t-s)^{\al-1}E_{\al,\al}\left(-\la^j_\ell(t-s)^\al\right)\pa_s^\al\psi_k(s)\,\rd s\right|^2\left|\left(G^j_k,\rho^j\vp_\ell^j\right)\right|^2\\
& \le C_k\left(\int_0^t(t-s)^{\al(1-\ga)-1}|\pa_s^\al\psi_k(s)|\,\rd s\right)^2\sum_{\ell=1}^\infty\left|\left(G^j_k,\rho^j\vp_\ell^j\right)\right|^2\\
& \le C_k\|\psi_k\|_{W^{2,\infty}(\BR)}^2t^{2(1-\ga)}\|\rho^j\|_{L^\infty(\Om)}\|G^j_k\|_{L^2(\Om)}^2.
\end{align*}
This proves that
\[
t\longmapsto\e^{-\xi t}w^j_k(t,\,\cdot\,)\in L^1(\BR_+;H^{2\ga}(\Om)),\quad\xi>0,
\]
and we deduce that
\[
t\longmapsto\e^{-\xi t}u^j_k(t,\,\cdot\,)\in L^1(\BR_+;H^{2\ga}(\Om)),\quad\xi>0.
\]
Using the continuity of $H^{2\ga}(\Om)\ni w\longmapsto\pa_\nu w|_{\Ga_\out}\in L^2(\Ga_\out)$, we deduce that
\[
t\longmapsto\e^{-\xi t}\pa_\nu u^j_k(t,\,\cdot\,)\in L^1(\BR_+;L^2(\Ga_\out)),\quad\xi>0,\ k\in\BN.
\]
Therefore, applying the Laplace transform $\cL$ in time on both sides of \eqref{t1e}, we obtain
\[
\La^1(\xi)b_k(\cL\psi_k)(\xi)\eta_k=\La^2(\xi)b_k(\cL\psi_k)(\xi)\eta_k,\quad\xi>0,\ k\in\BN,
\]
where
\[
\cL[\psi_k](\xi):=\int_0^\infty\e^{-\xi t}\psi_k(t)\,\rd t.
\]
Applying the fact that $\psi_k\ge0,\not\equiv0$, we deduce that $\cL[\psi_k](\xi)>0$ for all $\xi>0$. Then, using the fact $b_k>0$ and the linearity of $\La^j(\xi)$ ($j=1,2$), we get
\[
\La^1(\xi)h=\La^2(\xi)h,\quad\xi>0,\ \forall\,h\in\rspan\{\eta_k\}_{k\in\BN}.
\]
Finally, the density of $\rspan\{\eta_k\}$ in $H^{3/2}(\pa\Om)$ implies \eqref{t1i}.\smallskip

{\bf Step 3}\ \ In this step, by $\{\la^j_k\}_{k\in\BN}$ and $m^j_k\in\BN$ we denote the strictly increasing sequence of the eigenvalues of $\cA^j$ and the algebraic multiplicity of $\la^j_k$, respectively. For each eigenvalue $\la^j_k$, we introduce a family $\{\vp^j_{k,\ell}\}_{\ell=1}^{m^j_k}$ of eigenfunctions of $\cA^j$, i.e.,
\[
\cA^j\vp^j_{k,\ell}=\la^j_k\vp^j_{k,\ell},\quad\ell=1,\ldots,m^j_k,
\]
which forms an orthonormal basis in $L^2(\Om;\rho^j\,\rd\bm x)$ of the algebraic eigenspace of $\cA^j$ associated with $\la^j_k$. We fix also
\[
\Te^j_k(\bm x,\bm x'):=a^j(\bm x)a^j(\bm x')\sum_{\ell=1}^{m^j_k}\pa_\nu\vp^j_{k,\ell}(\bm x)\pa_\nu\vp^j_{k,\ell}(\bm x'),\quad\bm x,\bm x'\in\pa\Om,\ k\in\BN.
\]
In this step, we show that \eqref{t1d} implies
\begin{alignat}{2}
& \la^1_k=\la^2_k,\quad m^1_k=m^2_k, & \quad & k\in\BN,\label{t1j}\\
& \Te^1_k(\bm x,\bm x')=\Te^2_k(\bm x,\bm x'), & \quad & \bm x'\in\Ga'_\rin,\ \bm x\in\Ga_\out,\ k\in\BN.\label{t1k}
\end{alignat}
Taking the scalar product of $\vp^j_{k,\ell}$ with $U^j(\xi)$ solving \eqref{t1h}, one can check that
\[
U^j(\xi)=-\sum_{k=1}^\infty\f{\sum_{\ell=1}^{m^j_k}\left\langle\chi\,h,a^j\pa_\nu\vp^j_{k,\ell}\right\rangle\vp^j_{k,\ell}}{\xi^\al+\la^j_\ell}.
\]
Therefore, differentiating with respect to $\xi$ on both side yields
\[
\pa_\xi U^j(\xi)=\al\,\xi^{\al-1}\sum_{k=1}^\infty\f{\sum_{\ell=1}^{m^j_k}\left\langle\chi\,h,a^j\pa_\nu\vp^j_{k,\ell}\right\rangle\vp^j_{k,\ell}}{(\xi^\al+\la^j_\ell)^2}.\]
On the other hand, using the fact that
\[
\sum_{k=1}^\infty\left|\f{\sum_{\ell=1}^{m^j_k}\left\langle\chi\,h,a^j\pa_\nu\vp^j_{k,\ell}\right\rangle}{\xi^\al+\la^j_\ell}\right|^2=\|U^j(\xi)\|_{L^2(\Om;\rho^j\,\rd\bm x)}^2<\infty,
\]
we deduce that the sequence
\[
\sum_{k=1}^\infty\f{\sum_{\ell=1}^{m^j_k}\left\langle\chi\,h,a^j\pa_\nu\vp^j_{k,\ell}\right\rangle\vp^j_{k,\ell}}{(\xi^\al+\la^j_\ell)^2}
\]
converges in the sense of $\cD(\cA^j)$. Using the continuous embedding of $\cD(\cA^j)$ into $H^2(\Om)$ and the continuity of $H^2(\Om)\ni w\longmapsto\pa_\nu w|_{\Ga_\out}\in L^2(\Ga_\out)$, we deduce that
\begin{equation}\label{t1l}
\left.\pa_\xi\pa_\nu U^j(\xi)\right|_{\Ga_\out}=\left.\pa_\nu\pa_\xi U^j(\xi)\right|_{\Ga_\out}=\al\,\xi^{\al-1}\sum_{k=1}^\infty\f{\sum_{\ell=1}^{m^j_k}\left\langle\chi\,h,a^j\pa_\nu\vp^j_{k,\ell}\right\rangle\left.\pa_\nu\vp^j_{k,\ell}\right|_{\Ga_\out}}{(\xi^\al+\la^j_\ell)^2}.
\end{equation}
On the other hand, in view of \eqref{t1i}, we have
\[
\left.a^1\pa_\nu U^1(\xi)\right|_{\Ga_\out}=\left.a^2\pa_\nu U^2(\xi)\right|_{\Ga_\out},\quad\xi>0,\ h\in H^{3/2}(\pa\Om).
\]
Differentiating the above identity with respect to $\xi$ and applying \eqref{t1l}, we deduce
\[
\sum_{k=1}^\infty\f{\sum_{\ell=1}^{m^1_k}\left\langle\chi\,h,a^1\pa_\nu\vp^1_{k,\ell}\right\rangle\left.a^1\pa_\nu\vp^1_{k,\ell}\right|_{\Ga_\out}}{(\xi^\al+\la_\ell^1)^2}=\sum_{k=1}^\infty\f{\sum_{\ell=1}^{m^2_k}\left\langle\chi\,h,a^2\pa_\nu\vp^2_{k,\ell}\right\rangle\left.a^2\pa_\nu\vp^2_{k,\ell}\right|_{\Ga_\out}}{(\xi^\al+\la_\ell^2)^2}
\]
for $h\in H^{3/2}(\pa\Om)$ and $\xi>0$. Combining this identity with the arguments used in Step 4 in the proof of \cite[Theorem 2.2]{KOSY}, we deduce that \eqref{t1j}--\eqref{t1k} holds true.\smallskip

{\bf Step 4}\ \ We will complete the proof Theorem \ref{t1} in this step. Repeating some arguments used at the end of the proof of \cite[Theorem 1.1]{CK1}, we deduce that, for $j=1,2$, there exists an eigensystem $\{(\la^j_k,\vp^j_k)\}_{k\in\BN}$ of $\cA^j$ such that
\begin{equation}\label{t1m}
\la^1_k=\la^2_k,\quad a^1\pa_\nu\vp_k^1=a^2\pa_\nu\vp_k^2\mbox{ on }\pa\Om,\quad k\in\BN.
\end{equation}
Now let us recall the following inverse spectral result which follows from \cite[Corollaries 1.5--1.7]{CK2}.

\begin{lemma}\label{p1}
Under the conditions of Theorem \ref{t1}, assume that either of the three assumptions (i), (ii) or (iii) holds. Then \eqref{t1m} implies $(\rho^1,a^1,c^1)=(\rho^2,a^2,c^2)$.
\end{lemma}

Applying this result, we can complete Theorem \ref{t1}.
\end{proof}

\begin{proof}[Proof of Theorem \ref{t2}]
Repeating the arguments used in the proof of Theorem \ref{t1} and utilizing Proposition \ref{l2}, we deduce
\begin{equation}\label{t2c}
\pa_\nu v_k^1=\pa_\nu v_k^2\quad\mbox{on }\BR_+\times\pa\Om,\ k\in\BN,
\end{equation}
where $v_k^j$ ($j=1,2$, $k\in\BN$) solves \eqref{eq3} with $(\al,\rho,\bm B)=(\al^j,\rho^j,\bm B^j)$ ($j=1,2$). For $j=1,2$ and $\xi>0$, consider the boundary value problem
\[
\begin{cases}
(-\tri+\xi^{\al^j}\rho^j)V^j(\xi)+\bm B^j\cdot\nb V^j(\xi)=0 & \mbox{in }\Om,\\
V^j(\xi)=\phi & \mbox{on }\pa\Om,
\end{cases}\quad \phi\in H^{3/2}(\pa\Om).
\]
We associate this problem with the Dirichlet-to-Neumann map
\[
\cN^j(\xi):\phi\longmapsto\left.\pa_\nu V^j(\xi)\right|_{\pa\Om},\quad j=1,2,\ \xi>0.
\]
In a similar manner to the proof of Theorem \ref{t1}, we can prove that \eqref{t2c} implies
\begin{equation}\label{t2d}
\cN^1(\xi)=\cN^2(\xi),\quad\xi>0.
\end{equation}
Using the elliptic regularity of the operator $-\tri+\bm B^j\cdot\nb$, one can check the continuity of $[0,\infty)\ni\xi\longmapsto\cN^j(\xi)\in \cB(H^{3/2}(\pa\Om); H^{1/2}(\pa\Om))$ ($j=1,2$). Therefore, sending $\xi\to0$ in condition \eqref{t2d}, we obtain
\begin{equation}\label{t2e}
\left.\pa_\nu V^1(0)\right|_{\pa\Om}=\left.\pa_\nu V^2(0)\right|_{\pa\Om}
\end{equation}
for any $V(0)$ solving
\[
\begin{cases}
-\tri V(0)+\bm B^j\cdot\nb V(0)=0 & \mbox{in }\Om,\\
V(0)=\phi & \mbox{on }\pa\Om,
\end{cases}\quad \phi\in H^{3/2}(\pa\Om).
\]
Combining \eqref{t2e} with \cite[Theorem 1.1]{Po}, we deduce that $\bm B^1=\bm B^2$. Now choosing $\xi=1$ and applying \cite[Proposition 2.1]{Po} (see also \cite{KU,Sa} for equivalent results stated for magnetic Shr\"odinger operators), we deduce that $\rho^1=\rho^2$. Finally, choosing $\xi=\e$ and applying \cite[Proposition 2.1]{Po}, we get $\rho^1\exp(\al^1)=\rho^2\exp(\al^2)=\rho^1\exp(\al^2)$, indicating $\exp(\al^1)=\exp(\al^2)$ and thus $\al^1=\al^2$. This completes the proof of Theorem \ref{t2}.
\end{proof}

\begin{proof}[Proof of Corollary \ref{t4}]
For $j=1,2$, we define the operator $\cA^j$ as
\[
\cD(\cA^j)=H^2(\cM^j)\cap H^1_0(\cM^j),\quad\cA^j:=-\tri_{g^j,\mu^j}+c^j.
\]
Here we consider the weighted measure $\mu^j\,\rd V_{g^j}$, where $\rd V_{g^j}$ is the Riemannian volume measure associated with $(\cM^j,g^j)$. Parallel to Step 3 in the proof of Theorem \ref{t1}, we denote the eigensystem of $\cA^j$ ($j=1,2$) by $\{(\la^j_k,\{\vp^j_{k,\ell}\}_{\ell=1}^{m^j_k})\}_{k\in\BN}$, where $\{\la^j_k\}$ is strictly increasing and $m^j_k$ is the algebraic multiplicity of $\la^j_k$. Similarly, again $\{\vp^j_{k,\ell}\}_{\ell=1}^{m^j_k}$ forms an orthonormal basis in $L^2(\cM^j)$ associated with $\la^j_k$. Fixing
\[
\Te^j_k(\bm x,\bm x'):=\sum_{\ell=1}^{m^j_k}\pa_\nu\vp^j_{k,\ell}(\bm x)\pa_\nu\vp^j_{k,\ell}(\bm x'),\quad\bm x,\bm x'\in\pa\cM^1,\ k\in\BN
\]
and repeating the arguments used in the proof of Theorem \ref{t1}, we can prove that \eqref{t4d} implies the conditions \eqref{t1j}--\eqref{t1k}. Thus, in a similar way to the proof of \cite[Theorem 1.2]{KOSY}, we apply some results in \cite{KKL} to check that \eqref{t1j}--\eqref{t1k} imply that $(\cM^1,g^1)$ and $(\cM^2,g^2)$ are isometric and conditions \eqref{kappa}--\eqref{mu_c} are fulfilled. This completes the proof of Corollary \ref{t4}.
\end{proof}

\begin{proof}[Proof of Corollary \ref{t5}]
Using the notation of the proof of Corollary \ref{t4} with $\cA^j:=-\tri_{g^j}$ acting on $L^2(\cM^j)$ and $\cD(\cA^j):=H^1_0(\cM^j)\cap H^2(\cM^j)$, we can prove again that \eqref{t4d} implies \eqref{t1j}--\eqref{t1k}. For $j=1,2$, consider the initial-boundary value problem
\[
\begin{cases}
(\pa_t^2-\tri_{g^j})u^j=0 & \mbox{in }(0,\infty)\times\cM^j,\\
u^j=\pa_t u^j=0 & \mbox{in }\{0\}\times\cM^j,\\
u^j=\Phi & \mbox{on }(0,\infty)\times\pa\cM^j
\end{cases}
\]
and the associated hyperbolic partial Dirichlet-to-Neumann map
\[
\cN^j:C^\infty(\BR_+\times\Ga'_\rin)\ni \Phi\longmapsto\left.\pa_\nu u^j\right|_{(0,\infty)\times\Ga_\out}.
\]
Repeating some arguments used in \cite[Theorem 5.3]{KOSY}, we can show that \eqref{t1j}--\eqref{t1k} imply $\cN^1=\cN^2$. Thus, it follows from \cite[Theorem 1]{LO} that $(\cM^1,g^1)$ and $(\cM^2,g^2)$ are isometric. This completes the proof of Corollary \ref{t5}.
\end{proof}

\begin{proof}[Proof of Corollary \ref{t6}]
Using the notation of the proof of Corollary \ref{t4} with $\cA^j:=-\tri_g-c^j$ acting on $L^2(\cM^j)$ and $\cD(\cA^j):=H^1_0(\cM^j)\cap H^2(\cM^j)$, we can prove again that \eqref{t4d} implies \eqref{t1j}--\eqref{t1k}. 
For $j=1,2$, consider the initial-boundary value problem
\[
\begin{cases}
(\pa_t^2-\tri_{g^j}+c^j)u^j=0 & \mbox{in }(0,\infty)\times\cM^j,\\
u^j=\pa_t u^j=0 & \mbox{in }\{0\}\times\cM^j,\\
u^j=\Phi & \mbox{on }(0,\infty)\times\pa\cM^j
\end{cases}
\]
and the associated hyperbolic partial Dirichlet-to-Neumann map
\[
\cN^j:C^\infty(\BR_+\times\Ga'_\rin)\ni \Phi\longmapsto\left.\pa_\nu u^j\right|_{(0,\infty)\times\Ga_\out}.
\]
Similarly to the proof of Corollary \ref{t5}, we deduce from \eqref{t1j}--\eqref{t1k} that $\cN^1=\cN^2$. Thus, it follows from \cite[Theorem 1]{KKLO} that there exists a neighborhood $U$ of $\Ga_\out$ in $\cM$ such that $c^1=c^2$ in $U$. This completes the proof of Corollary \ref{t6}.
\end{proof}

%

\section*{Acknowledgments}

This work is partly supported by Japan Society for the Promotion of Science (JSPS) KAKENHI Grant Number JP15H05740. The first author is partially supported by  the French National
Research Agency ANR (project MultiOnde) grant ANR-17-CE40-0029.
The second, third and fourth authors are partially supported by the A3 Foresight Program ``Modeling and Computation of Applied Inverse Problems'', JSPS.
The second author is supported by National Natural Science Foundation of China (NSFC) (No.\! 11801326).
The fourth author is partly supported by NSFC (No.\! 91730303) and RUDN University Program 5-100.


\bibliographystyle{unsrt}

\end{document}